\documentclass[a4paper,12pt,english]{amsart}

\usepackage{amsmath,upref,amssymb}
\usepackage{tensor}

\usepackage{color}

\swapnumbers

\input xypic
\xyoption{all}
\xyoption{arc}
\xyoption{2cell}
\CompileMatrices

\RequirePackage{calrsfs}
\DeclareSymbolFont{rsfscript}{OMS}{rsfs}{m}{b}
\DeclareSymbolFontAlphabet{\mathrsfs}{rsfscript}
\renewcommand{\mathcal}{\mathrsfs}

\newcommand{\DMO}{\DeclareMathOperator}
\numberwithin{equation}{subsection}

\swapnumbers

\theoremstyle{plain} 

\newtheorem{prop}[subsection]{Proposition}
\newtheorem{lemma}[subsection]{Lemma}

\theoremstyle{definition}

\theoremstyle{remark}
\newtheorem{remark}[subsection]{Remarks}

\newtheorem*{case*}{Case}

\makeatletter
\renewcommand{\p@enumi}{\thesubsection}
\renewcommand{\p@enumii}{\thesubsection\theenumi}
\makeatother

\newenvironment{num}{\renewcommand{\theenumi}{(\alph{enumi})}
                      
                                          \begin{enumerate} }
                    {\end{enumerate} }
 \newenvironment{conds}{\renewcommand{\theenumi}{(\roman{enumi})}
                       
                        \begin{enumerate} }
                     {\end{enumerate} }

\newcommand{\be}{\begin{equation}}
\newcommand{\ee}{\end{equation}}

\newcommand{\bee}{\begin{equation*}}
\newcommand{\eee}{\end{equation*}}

\newcommand{\bs}{\begin{split}}
\newcommand{\es}{\end{split}}

\newcommand{\bc}{\begin{cases}}
\newcommand{\ec}{\end{cases}}

\newcommand{\bml}{\begin{multline}}
\newcommand{\eml}{\end{multline}}

\newcommand{\bmll}{\begin{multline*}}
\newcommand{\emll}{\end{multline*}}

\newcommand{\belb}[1]{\mar{#1}\begin{equation}\label{#1}}

\newcommand{\mpair}[1]{\pair{\,#1\,}}

\newcommand{\mset}[1]{\set{\,#1\,}}

\newcommand{\pair}[1]{\langle #1\rangle}

\newcommand{\set}[1]{\{#1\}}

\newcommand{\CC}{{\mathcal C}}

\def\a{\alpha}
\def\b{\beta}
\def\g{\gamma}
\def\G{\Gamma}
\def\d{\delta}
\def\D{\Delta}

\def\L{\Lambda}

\def\s{\sigma}

\def\th{\theta}

\def\t{\tau}

\def\to{\rightarrow}

\def\lsupp#1#2{\kern\scriptspace\vphantom{#2}^{#1}\kern-\scriptspace#2}

\def\lsubb#1#2{\kern\scriptspace\vphantom{#2}_{#1}\kern-\scriptspace#2}

\newcommand{\seq}{{\,\subseteq\,}}

\newcommand{\sneq}{{\,\subsetneq\,}}

\newcommand{\sm}{{\,\setminus\,}}

\newcommand{\eset}{{\emptyset}}

\newcommand{\ck}[1]{{#1}^\vee}

\newcommand{\wh}{\widehat}

\def\dotcup{\hskip1mm\dot{\cup}\hskip1mm}

\newcommand{\mar}{\marginpar}

\DeclareMathOperator{\supp}{{\mathrm{supp}}}

\DMO{\Supp}{{\mathrm{Supp}}}

\DMO{\Rad}{{\mathrm{Rad}}}

\DMO{\Ann}{{\mathrm{Ann}}}

\DMO {\Add}{{\mathrm{Add}}}

\DMO {\add}{{\mathrm{add}}}

\DMO {\Img}{{\mathrm{Im}}}

\DMO {\coim}{{\mathrm{Coim}}}

\DMO {\coker}{{\mathrm{Coker}}}

\DMO {\colim}{\varinjlim}

\DMO {\plim}{\varprojlim}

\DMO {\mEnd}{{\mathrm{End}}}

\DMO {\mend}{{\mathrm{end}}}

\DMO {\Proj}{{\mathrm{Proj}}}

\DMO {\Ext}{{\mathrm{Ex}t}}

\DMO {\ext}{{\mathrm{ext}}}

\DMO {\tor}{{\mathrm{tor}}}

\DMO {\Tor}{{\mathrm{Tor}}}

\DMO {\Hom}{{\mathrm{Hom}}}

\DMO {\HOM}{{\mathrm{HOM}}}

\DMO {\Modfg}{{\mathrm{-Modfg}}}

\DMO {\modulfg}{{\mathrm{-modfg}}}

\DMO {\modul}{{\mathrm{-mod}}}

\DMO {\Mod}{{\mathrm{-Mod}}}

\DMO {\pdim}{{\mathrm{proj.dim.}}}

\DMO {\PD}{{{\mathrm{Proj.Dim.}}}}

\DMO {\gldim}{{\mathrm{gr.gl.dim.\ }}}

\DMO {\grad}{{\mathrm{rad} }}

\DMO {\Sheaves}{{\mathrm{Sh}}}

\DMO {\Flab}{{\mathrm{Fl}}}

\DMO {\Poinc}{{\mathrm{Poinc}}}

\DMO {\Groth}{{{{\mathrm{K}}}_0}}

\DMO {\Mat}{{\mathrm{Mat}}}

\DMO \Sl{{\mathrm{sl}}}

\DMO \SL{{\mathrm{SL}}}

\DMO \Gl{{\mathrm{gl}}}

\DMO \GL{{\mathrm{GL}}}

\DMO \lcm{\mathrm{{lcm}}}

\DMO \rank{{\mathrm{rank}}}

\DMO \diag{{\mathrm{diag}}}

\DMO\vtx{{\mathrm{Vert}}}

\DMO\pc{{\mathrm{ParClos}}}

\DMO {\Addp}{{\mathrm{Add}}'}

\DMO {\addp}{{\mathrm{add}}'}

\DMO {\hatGr}{{\widehat{{\mathrm{K}}}_0}}

\DMO {\NExt}{{\mathrm{NExt}}}

\DMO {\nExt}{{\mathrm{next}}}

\DMO {\DExt}{{\mathrm{DEx}t}}

\DMO {\dext}{{\mathrm{dext}}}

\newcommand{\Int}{{\mathbb Z}}
\newcommand{\Nat}{{\mathbb N}}
\newcommand{\real}{{\mathbb R}}
\newcommand{\rat}{{\mathbb  Q}}

\DMO{\ob}{ob}
\DMO{\mor}{mor}
\DMO{\tr}{tr}
\DMO{\spec}{spec}

\DMO{\cov}{cov}
\DMO{\Sym}{Sym}
\DMO{\Dih}{Dih}
\DMO{\Spec}{Spec}
\DMO{\domn}{dom}
\DMO{\cod}{cod}
\DMO{\ord}{ord}

\DMO{\stab}{{\mathrm{Stab}}}
\DMO{\inter}{{\mathrm{Int}}}
\DMO{\coeff}{{\mathrm{Coeff}}}
\DMO{\conv}{{\mathrm{Conv}}}
\begin{document}

\title[Parabolic subgroup orbits]{Parabolic subgroup orbits on finite root systems}

\author{M. J.  Dyer and G. I.  Lehrer}
\address{Department of Mathematics 
\\ 255 Hurley Building\\ University of Notre Dame\\
Notre Dame, Indiana 46556, U.S.A.}
\email{dyer@nd.edu}
\address{School of Mathematics and Statistics, University of Sydney\\
 Sydney, NSW. 2006\\
 Australia.}
\email{gustav.lehrer@sydney.edu.au}


\date{\today}
\begin{abstract}  Oshima's Lemma describes the orbits of parabolic subgroups of irreducible finite Weyl groups on  crystallographic  root systems. 
This note generalises that result to all root systems of finite Coxeter groups, and provides a self contained proof, independent of the representation
theory of semisimple complex Lie algebras. 
\end{abstract}

\maketitle
%
%
%
%
\section*{Introduction}\label{s0}
This note concerns a lemma (Lemma \ref{lem:oshorig} below) of Oshima   on orbits of 
standard parabolic subgroups of an irreducible finite Weyl group acting on its 
crystallographic root system. That result is proved in \cite[Lemma 4.3]{Osh} 
using results from \cite{OdaOsh} about  the representation theory of  semisimple 
complex Lie algebras. It is applied  in  \cite{Osh}, and plays a crucial role in \cite{DyLFund}, in the
 study of conjugacy classes of reflection subgroups of finite Weyl groups.

The purpose of  this note is to give several  reformulations and generalisations  of the lemma
which apply to general root systems of finite Coxeter groups and to provide 
an elementary self-contained proof of the generalised lemma. Though the 
statements are uniform, the proof involves some casewise analysis.

This note is organised as follows. In Section \ref{s1}, we
 briefly state Oshima's lemma (Lemma \ref{lem:oshorig}) and our 
 generalisation (Proposition \ref{prop:osh}).  Section 2 gives more details on the notions 
 involved in the formulation of these results and their proofs. Section \ref{s3} reduces the proofs of both results  to showing that 
Proposition  \ref{prop:osh}(a) holds   for at least one root system 
of each irreducible finite Coxeter group
and proves \ref{prop:osh}(a) for  dihedral 
groups.  Section \ref{s4} proves \ref{prop:osh}(a) for 
crystallographic root systems of finite Weyl groups, using 
elementary properties of root strings. By the classification of 
irreducible finite Coxeter groups, the proof of \ref{prop:osh}(a) in general is 
reduced to the cases of $W$ of type $H_{3}$  and $H_{4}$; these 
are treated in Section \ref{s5} by using foldings of Coxeter 
graphs to reduce to types $D_{6}$ and $E_{8}$, where the result 
is known from Section \ref{s4}. 


\section{Orbits of parabolic subgroups on finite root systems}
\label{s1}
\subsection{Oshima's lemma}  \label{oshorig} Let $\Phi\subseteq V$ be a root system 
of a finite Coxeter  system $(W,S)$,  and  let $\Pi$ be a set of 
simple roots corresponding to  $S$ (see Section \ref{s2} for details on terminology). For  $\a\in \real\Phi$
and $\b\in \Pi$, define the \emph{root coefficient} $\a[\b]\in \real$  by  $\a=\sum_{\b\in \Pi}\a[\b]\b$; 
define  the \emph{support} of $\a$ to be  $\supp(\a):=\mset{\b\in \Pi\mid \a[\b]\neq 0}$. 

The following result is proved in \cite[Lemma 4.3]{Osh} using facts from \cite{OdaOsh} about  the representation theory of  semisimple complex Lie algebras. 
\begin{lemma} \label{lem:oshorig} Assume above  that  $W$ is an irreducible finite Weyl group and  $\Phi$ is  crystallographic.   
 Fix $\Delta\seq \Pi$, scalars $c_{\b}\in \real $ for $\b\in \Delta$, not all zero, and $l\in \real$. Let 
  \begin{equation*}
X:=\mset{\g\in \Phi\mid \mpair{\g,\g}^{1/2}=l \text{ \rm and } \g[\b]=c_{\b} \text{ \rm for all $\b\in \D$}}.
\end{equation*}  If $X$ is non-empty, then it is   a single   $W_{\Pi\sm \D}$-orbit of roots where $W_{\Pi\sm \D}$ is the standard parabolic subgroup with simple roots $\Pi\sm \D$.  Equivalently,   $\vert X\cap \CC_{\Pi\sm \D}\vert \leq 1$  where  $\CC_{\Pi\sm \D}$ is the closed  fundamental chamber
of $W_{\Pi\sm \D}$.
\end{lemma} 

\subsection{} \label{osh}   The result below gives  reformulations of
Lemma  \ref{lem:oshorig}  which are valid for all (possibly non-crystallographic or reducible)  finite
Coxeter groups. 
\begin{prop}\label{prop:osh} Let $\Phi$ be a root system of a finite Coxeter group $W$, with simple roots $\Pi$ and  corresponding simple reflections $S$. For $J\subseteq S$, let $W_{J}$ denote the standard parabolic subgroup of $W$ generated by $J$, and $\mathcal{C}_{W_{J}}$ denote its closed fundamental chamber.
\begin{num}
\item Let $J\seq S$. If $\a\in \Phi\sm \Phi_{J}$, then $W\a\cap (\a +\real \Pi_{J})=W_{J}\a$. 
\item Let $J\seq S$. If $\a,\b\in (\Phi\sm \Phi_{J})\cap \CC_{W_{J}}$ and $\b\in W\a\cap (\a +\real \Pi_{J})$, then $\a=\b$. 
 \item Let $\Phi'$ be the 
root system  of a parabolic subgroup $W'$ of $W$. Then
for any $\b\in \Phi\sm \Phi'$, one has $W\b\cap(\b+ \real \Phi')=W'\b$.
\end{num}\end{prop}
\begin{remark}

(1) The assumption in   (a) that $\a \in \Phi\sm \Phi_{J}$  cannot be  weakened to  $\a\in \Phi$. For take  $\Phi$ of type $A_{3}$ in $\real^{4}$ with simple roots $\a_{i}:=e_{i}-e_{i+1}$ for $i=1,\ldots, 3$, $\Pi_{J}=\set{\a_{1},\a_{3}}$ and $\a=\a_{1}$. Then $\a_{3}\in  \Phi\cap (\a_{1}+\real \Pi_{J})$ but $\a_{3}\not \in W_{J}\a_{1}=\set{\pm \a_{1}}$.

(2) The proposition  does not extend as stated to  infinite Coxeter groups. 
For example, take $(W,S)$ irreducible affine of type $\widetilde  A_{3}$ with $\Pi=\set{\a_{0},\a_{1},\a_{2},\a_{3}}$ where 
$\a_{i}$ and $\a_{j}$ are joined in the Coxeter graph if $i-j\in \set{\pm 1} \pmod 4$.
Let
$\delta:=\a_{0}+\a_{1}+\a_{2}+\a_{3}$ denote  the standard indivisible isotropic root 
(which is not in  $\Phi$; see \cite{K}). Take $\Pi_{J}=\mset{\a_{1},\a_{3}}$. Let $\a=2\d+\a_{1}\in \Phi\sm \Phi_{J}$. Then
$2\d+\a_{3}\in W\a\cap ( \a+\real\Pi_{J})$ but $2\d+\a_{3} \not\in W_{J}\a$ by (1), 
since $W$ fixes $\d$.

(3) It may also be shown that if $v\in \CC_{W}$, $w\in W$ and $J\seq S$,
then $Wv\cap (v+\real \Pi_{J})=W_{J}v$. This statement generalises to possibly infinite Coxeter groups; see \cite[Lemma 2.4(d)]{dyer:imc}.

(4)  We do not know any natural result  which generalizes both  \ref{prop:osh}(a) and (3). Also, 
there is no known extension of  \ref{prop:osh}(a) or (3) to unitary reflection groups.
  \end{remark}

\section{Preliminaries}
\label{s2} The remainder of the paper gives a proof of Lemma \ref{lem:oshorig} and Proposition \ref{prop:osh}. This preliminary  section 
sets out more precisely the notation, terminology 
and background required for the statements of these results and their proofs.
\subsection{}  Let  \label{rootsyst} $V$ be a real Euclidean space i.e. a finite dimensional 
real vector space equipped with a symmetric, positive definite bilinear form $\mpair{-,-}\colon V\times V\to \real$. 
 For non-zero  $\a\in V$,  let $s_{\a}\colon V\to V$ denote the orthogonal reflection in $\a$; 
 it is the $\real$-linear map  defined by $s_{\a}(v)=v-\mpair{v,\ck \a}\a$ where $\ck \a:=\frac{2}{\mpair{\a,\a}}\a$. 
 In this paper, by  a   \emph{root system} $\Phi$  in $V$, we shall mean  a subset $\Phi$ of $V$ 
satisfying the conditions (i)--(iii) below:
\begin{conds}
\item $\Phi$ is a finite subset of $V\sm\set{0}$
\item If $\a,\b\in \Phi$, then $s_{a}(\b)\in \Phi$.
\item If $\a,c\a\in \Phi$ with $c\in \real$, then $c\in \set{\pm1}$.\end{conds}
This includes all the  most commonly used notions of roots systems of finite Coxeter groups, except for  non-reduced crystallographic root systems of finite Weyl groups.

The subgroup $W$ of $\mEnd(V)$ generated by $\mset{s_{\a}\mid \a\in \Phi}$ is a finite (real) reflection group i.e. a finite  Coxeter group.
A \emph{simple system} $\Pi$ of $\Phi$ is a linearly independent  subset $\Pi\seq \Phi$  such that 
$\Phi=\Phi_{+}\dotcup \Phi_{-}$ where $\Phi_{+}:=\Phi\cap \real_{\geq 0}\Pi$ 
and $\Phi_{-}=-\Phi_{+}$. Fix a simple system $\Pi$ (it is well  known that such 
simple systems exist). Then $\Phi_{+}$ is the corresponding positive system of $\Phi$  
and  $S:=\mset{s_{\a}\mid \a
\in \Pi}\seq W$ is called the set of \emph{simple reflections} of $W$. It is well 
known that $(W,S)$ is a Coxeter system.  The subset 
$T:=\mset{s_{\a}\mid \a\in \Phi}=\mset{wsw^{-1}\mid w\in W, s\in S}$  of $W$ 
 is called the set of reflections of $W$.
 
 The definition of root coefficients and support of roots from
 \ref{oshorig} extends without change to the broader class of root systems considered in this section.

\subsection{Dual root system} \label{dualrootsys} If $\Phi$ is a root system in 
$V$, then $\ck \Phi:=\mpair{\ck\a\mid \a\in \Phi}$ is a root system, called the 
\emph{dual root system} of $\Phi$; it has a system of simple roots $\ck \Pi:=\mset {\ck \a
\mid \a\in \Pi}$ with corresponding positive roots 
$\ck \Phi_{+}:=\mset{\ck \a\mid \a\in \Phi_+}$ and 
associated finite  Coxeter system $(W,S)$.
\subsection{Weyl groups} \label{crystrootsys} 
The root system $\Phi$ is said to be crystallographic
if for all $\a,\b\in \Phi$, one has $\mpair{\a,\ck\b}\in \Int$.  In that case, $W$ is a 
finite Weyl group and $\Phi$ (considered as a subset of its linear span) may be regarded as  a reduced root system in the sense of \cite[Ch VI]{Bour}.
  
\subsection{Reflection subgroups}\label{refsubgp} A subgroup $W'$  of $V$ generated by a subset  of $T$ 
is called a reflection subgroup.
It has a root system $\Phi_{W'}=\mset{\a\in \Phi\mid s_{\a}\in W'}$ and a 
  unique simple system $\Pi_{W'}\seq \Phi_+$; the corresponding positive 
  system $\Phi_{W',+}$ is $\Phi_{+}\cap \Phi_{W'}$.  One has $\Phi_{W'}=W'\Pi_{W'}$.
  
 The reflection subgroups $W_{I}:=\mpair{I}$ generated by subsets $I$ of $S$ are 
 called \emph{standard parabolic subgroups}  and their conjugates are called 
 \emph{parabolic subgroups}.  If $W'=W_{I}$, then
 $\Pi_{W'}=\mset{\a\in \Pi\mid s_{a}\in I}$ and $\Phi_{W'}=\Phi\cap \real_{\geq 0}\Pi_{W'}$.
 
 \subsection{Fundamental chamber for the $W$-action on $V$}\label{ss:fundcham}
 The subset  $\CC=\CC_{W}:=\mset{v\in V\mid \mpair{v,\Pi}\seq \real_{\geq 0}}$ of $V$ is 
 called the (closed) \emph{fundamental chamber} of $W$. For a reflection subgroup $W'$, we denote by $\mathcal{C}_{W'}$ the  fundamental chamber for $W$ with respect to simple roots $\Pi_{W'}$. In the following Lemma, we collect some
 standard facts concerning this situation, which may be found in \cite{Bour, StYale}.

 \begin{lemma}\label{lem:fundcham}\begin{num}\item  Every $W$ orbit on $V$ contains a unique point of $\CC$.
 \item For $v\in \CC$, the stabiliser $\stab_{W}(v):=\mset{w\in W\mid w(v)=v}$ satisfies  $\stab_{W}(v)=W_{I}$ where 
 $I:=\mset{s\in S\mid s(v)=v}=\mset{s_{\a}\mid \a\in \Pi \cap v^{\perp}}$.
 \end{num}\end{lemma}
 
 It follows that if $\a\in \Phi$ is any root, then  the $W$-orbit $W\a$ contains 
 a unique element of $\CC_{W}$. A root  $\a\in \Phi\cap \CC_{W}$ is called 
 a  \emph{dominant root}.  If $\Phi$ is irreducible, there are at most two dominant roots
 (cf. \cite[Lemma 2]{DyLeRef}), and hence at most two $W$-orbits of
 roots.   If  $\Phi$ is crystallogrphic, the 
dominant  roots  are the highest long and short roots of the components 
of $\Phi$.

\section{Some reductions, and proof of \ref{prop:osh}(a) for dihedral groups}\label{s3}
This section provides  some reductions 
which are useful for the  proofs of Proposition \ref{prop:osh} and Lemma \ref{lem:oshorig}. 
These two results are then proved for dihedral reflection groups. 

\subsection{Reductions}\label{oshequiv} Suppose we are given positive scalars $d_{\g}$ for $\g\in \Phi$ such that $d_{\g}=d_{\g'}$ if 
$\g$ and $\g'$ are in the same $W$-orbit.  Since $s_{d_{\g}\g}=s_{\g}$, it follows that $\Psi:=\mset{d_{\g}\g\mid \g\in \Phi}$ is a 
root system for $(W,S)$, with $\mset{d_{\g}\g\mid \g\in \Pi}$ as a set of simple roots corresponding to $S$ and positive system 
$\Psi_{+}=\mset{d_{\g}\g\mid \g\in \Phi_{+}}$. We say that such a root system $\Psi$ is obtained by \emph{rescaling} $\Phi$. 
For example, the dual root system $\ck \Phi$ is obtained by rescaling $\Phi$. Any  root system of the finite reflection group
 $W$ on $V$ can be obtained from any other root system of $W$ by rescaling.

\begin{lemma}\label{lem:oshequiv}
\begin{num}
\item $\text{\rm Proposition \ref{prop:osh}(a)}$ implies $\text{\rm Lemma \ref{lem:oshorig}}$. 
\item For fixed  $(W,S,\Phi,\Pi)$, $\text{\rm Proposition \ref{prop:osh}(a)}$  is equivalent to 
$\text{\rm Proposition \ref{prop:osh}(b)}$.
\item $\text{\rm Proposition \ref{prop:osh}}$ follows from  the  special case of its part $\text{\rm (a)}$ in which  $\Phi$ is assumed to be  irreducible.
\item Let $\Psi=\set{d_{\g}\g\mid \g\in \Phi}$ be a rescaling of $\Phi$   as above. Then $\text{\rm Proposition \ref{prop:osh}(a)}$ holds for 
$\Phi$, $J$, and $\a$ if and only if it holds for $\Psi$, $J$ and $d_{\a}\a$.
\item If \text{ \rm Proposition \ref{prop:osh}(a)} holds for $(W,S,\Phi,\Pi)$, then for any $K\seq S$,  it holds for $(W_{K},K,\Phi_{K},\Pi_{K})$. 
 \end{num}
\end{lemma}

\begin{proof}  To prove (a), we assume $\text{\rm Proposition \ref{prop:osh}(a)}$ and make the assumptions of Lemma \ref{lem:oshorig}. 
Let $J\seq S$ with $\Pi_{J}=\Pi\sm \D$. Note that  for $\a\in \Phi$,  $\Phi\cap (\a+\real\Pi_{J})=\mset{\b\in \Phi\mid  
\b[\g]=\a[\g] \text{ \rm for all }\g\in \D}$. Also, \cite[Ch VI, \S 1, Proposition 11]
{Bour} implies that  $W\a=\mset{\g\in \Phi\mid \mpair{\g,\g}^
{1/2}=\mpair{\a,\a}^{1/2}}$.  Assume that $X\neq \eset$, say $\a\in X$.
Then the above implies that $X=W\a\cap (\a+\real\Pi_{J})$. By Proposition \ref{prop:osh}(a), $X=W_{J}\a$ is a single $W_{J}$-orbit, which proves (a).

Next, we prove (b). Let $J,\a,\b$ be as in  
Proposition \ref{prop:osh}(b).  If   Proposition \ref{prop:osh}(a) 
holds, it implies that $\b\in W_{J}\a$ and so $\b=\a$ since $\a,\b\in \CC_{W_{J}}$. Conversely, suppose that Proposition \ref{prop:osh}(b) holds for $(W,S,\Phi,\Pi)$ and let $J,\a$ be as in Proposition \ref{prop:osh}(a).
Obviously $W_{J}\a\seq W\a\cap (\a+\real\Pi_{J})$. To establish the reverse inclusion,  let $\b\in W\a\cap (\a+\real\Pi_{J})$. Choose $\a'\in W_{J}\a\cap \CC_{W_{J}}$ and $\b'\in
W_{J}\b\cap \CC_{W_{J}}$. Clearly, $\a',\b, \b'\in \a+\real\Pi_{J}$ and since $\a\not\in \Phi_{J}$, one has $\a',\b,\b'\in \Phi\sm \Phi_{J}$.  Proposition \ref{prop:osh}(b)  applied to $\a',\b'$ gives that $\a'=\b'$ and hence $\a\in W_{J}\a'=W_{J}\b'=W_{J}\b$. 

We next  show that Proposition \ref{prop:osh}(a) implies 
Proposition \ref{prop:osh}(c). Let $W'$, $\Phi$, $ \b$ be as in Proposition \ref{prop:osh}(c). Write 
$W'=wW_{J}w^{-1}$ where $w\in W$ and $J\seq S$, and  let $\a:=w^{-1}(\b)$. Then 
$\Phi'=w\Phi_{J}$ and so $\real \Phi'=w^{-1}\real \Phi_{J}=w^{-1}\real \Pi_{J}$
and $\a\in \Phi\sm  \Phi_{J}$.  Thus, using Proposition \ref{prop:osh}(a)
\begin{equation*}
W\b\cap (\b+\real\Phi')=Ww\a\cap (w\a+w\real \Pi_{J})=w\bigl(W\a\cap (\a+\real
\Pi_{J})=wW_{J}\a=W'\b\end{equation*} as required for Proposition \ref{prop:osh}(c). 
To complete the proof of (c), we show that  Proposition \ref{prop:osh}(a) follows from 
its special case in which $\Phi$ is irreducible.  Let $\a\in \Phi$.  It is known 
that $\supp(\a)$  is a non-empty connected subset  of $\Pi$, and it is therefore 
contained in some connected component $\G$ of $\Pi$. Let $K\seq S$ with $\Pi_{K}=\G$. 
Equivalently, $W_{K}$ is the unique component of $W$ containing $s_{\a}$. 
Then $\a\in \Phi_{K}$, $W\a=W_{K}\a$ and $W_{J}\a=W_{J\cap K}\a$. 
  Assume  that $\a\in \Phi\sm \Phi_{J}$. We claim that   $\Phi\cap(\a+\real \Pi_{J})=\Phi_{K}\cap (\a+\real\L\cap  \G) $. 
  Obviously, the right hand side is contained in the left. On the other hand, let $\beta\in \Phi\cap(\a+\real \Pi_{J})$. Since
   $\a\not\in \Phi_{J}$, we have   $\eset\neq \supp(\a)\sm \Pi_{J}= \supp(\b)\sm \Pi_{J}$ 
   and therefore $\b\in \Phi_{K}$ since $\supp(\b)$ is connected. Hence  $\b-\a\in\real \Pi_{K}\cap \real \Pi_{J}=\real\Pi_{K\cap J}$, 
   proving the claim.     If Proposition \ref{prop:osh}(a) holds when $W$ is irreducible, 
   we therefore have (by its validity for $W_{K}$)
  \begin{equation*} 
  W\a\cap (\a +\real \Pi_{J})=W_{K}\a\cap (\a +\real\L\cap  \G)=W_{J\cap K}\a=W_{J}\a 
  \end{equation*} 
  and Proposition \ref{prop:osh}(a) holds for arbitrary $W$.  Hence (c) is proved.
  
The straightforward verifications of (d) and  (e) are left to the reader.
\end{proof}

\subsection{Dihedral case} \label{oshdihed}  The remainder of the paper will focus on the proof of special cases of
Proposition \ref{prop:osh}(a) which together suffice to prove the result in general, by the preceding reductions. 
The following result disposes  of the case of  dihedral Coxeter systems $(W,S)$.   
\begin{lemma} \label{lem:oshdihed}
\begin{num}
\item Suppose that  $v\in V$, $\a\in \Phi$ and $w\in W$ with $wv\in v+\real \a$.
Then  $wv\in W'v$ where $W':=\mpair{s_{\a}}$.
\item  $\text{\rm Proposition \ref{prop:osh}(a)}$ holds if    $\vert S\vert \leq 2$.
 \end{num}
\end{lemma}
\begin{proof}
The proof of (a) is from the proof of \cite[Lemma 2]{CarConj}.
Write $wv=v+c\a$ where $c\in \real$.  If $c=0$, then $wv=v\in W'v$. 
Otherwise, a simple computation starting from the fact  that $\mpair{wv,wv}=\mpair{v,v}$ shows that $c=-\mpair{v,\ck\a}$ and so $wv=s_{\a}v\in W'v$.

For (b), assume $\vert S\vert \leq 2$ and $J\seq S$.
In this case, the statement of Proposition \ref{prop:osh}(a) is 
vacuous if $J=S$, trivial if $J=\eset$,  and immediate from (a) if  $\vert J\vert =1$.
\end{proof}
\section{Proof of \ref{prop:osh}(a) for finite Weyl groups}\label{s4}
Assume in this subsection  that $\Phi$  is crystallographic. It is  a reduced root system 
(in its linear span $\real\Phi$) in the sense of \cite{Bour}, 
with simple system $\Pi$, positive system $\Phi_{+}$ and  Weyl group $W$. 
The arguments below leading to a proof of Proposition \ref{prop:osh} in this case use   basic properties of root  strings,  as   in \cite[Ch VI, \S1]{Bour}. 

 \begin{lemma} \label{lem:rootstring} Let $\a_{1},\ldots, \a_{n}$ be roots in $\Phi$, where $n>0$. For $I\seq L:=\set{1,\ldots, n}$, let $\a_{I}:=\sum_{i\in I}\a_{i}$. Assume that $\a_{I}\neq 0$ for all $I\seq L$ with $I\neq \eset$  and that $\a_{L}=\a_{1}+\ldots +\a_{n}\in \Phi$.
 \begin{num}
 \item There is a permutation $\s$ of $\set{1,\ldots, n}$ such that for each
 $i=1,\ldots, n$, one has $\a_{\s(1)}+\ldots+\a_{\s(i)}\in \Phi$. 
 \item Suppose that $n=3$. If $\a_{1}+\a_{2} \in \Phi$ but $\a_{2}+\a_{3}\not\in \Phi$, then
 $\a_{1}+\a_{3}\in \Phi$.
 \item Assume that $\a_{i}+\a_{j}\not\in \Phi$ for any $i,j$ with $2\leq i<j\leq n$.
 Then $\a_{I}\in \Phi$ for all $I\seq L$ with $1\in I$. 
 \end{num}
 \end{lemma}
 \begin{remark}
 (1) For $m\in \Nat$, let 
 $c_{m}:=\vert \mset{I\seq J\mid \a_{I}\in \Phi, \vert I\vert =m}\vert $.
 It is shown in \cite{PF} that  $c_{i}\geq n-i+1$ for $i=1,\ldots, n$
 (this  follows from (a)--(b) if $n=3$). The  case $c_{i}=n-i+1$ for all 
 $i$ occurs in type $A$.  
 
(2)  Part (b)  is the key point for the proofs given here 
of Proposition \ref{prop:osh}. In the (present) crystallographic case, it
has a standard proof using the well-known  formulae
  ($[\mathfrak{g}_{\a}, \mathfrak{g}_{\b}]=\mathfrak{g}_{\a+\b}$ for $\a,\b\in \Phi$ with $\a+\b\in \Phi$ and $[\mathfrak{g}_{\a}, \mathfrak{g}_{\b}]=0$ for $\a,\b\in \Phi$ with  $\a+\b\not\in \Phi\cup\set{0}$) for brackets of root spaces $\mathfrak {g}_{\a}$, where  $\a\in \Phi$, of a  
  semi-simple complex Lie algebra $\mathfrak{g}$ with root system $\Phi$. In fact, it follows by computing \begin{equation*} [\mathfrak{g}_{\a_{2}},[\mathfrak{g}_{\a_{3}},
  \mathfrak{g}_{\a_{1}}]]= [\mathfrak{g}_{\a_{3}},[\mathfrak{g}_{\a_{2}},\mathfrak{g}_{\a_{1}}]]= [\mathfrak{g}_{\a_{3}},\mathfrak{g}_{\a_{1+}\a_{2}}]=\mathfrak{g}_{\a_{1}+\a_{2}+\a_{3}}\neq 0
  \end{equation*} using these formulae and the Jacobi identity. 
  We give another proof below.
  \end{remark}
  
\begin{proof} It is well known that part (a) may be  proved in the same way as its special case \cite[Ch VI, \S 1, Proposition 19]{Bour}
where all $\a_{i}\in \Phi_{+}$.  
      A proof of (b)   in terms of root strings is as follows.  
   Assume to the contrary that  $a_{1}+\a_{2}+\a_{3}\in \Phi$, $\a_{1}+\a_{2}\in \Phi$,
   but $\a_{2}+\a_{3}\not\in \Phi$ and $\a_{1}+\a_{3}\not\in \Phi$.
   Since $\a_{1},\a_{3}\in \Phi$ but $\a_{1}+\a_{3}\not\in \Phi\cup\set{0}$, one 
   gets $\mpair{\a_{3},\ck\a_{1}}\geq 0$  from \cite[Ch VI, \S 1, Corollary of Theorem 1]{Bour}.
   Similarly, one has $\mpair{\a_{3},\ck\a_{2}}\geq 0$ (from $\a_{2}+\a_{3}\not\in \Phi\cup\set
   {0}$),
   $\mpair{\a_{1}+\a_{2}+\a_{3},\ck \a_{1}}\leq 0$ (from $(\a_{1}+\a_{2}+\a_{3})-\a_{1}\not\in 
   \Phi\cup\set{0}$) and  $\mpair{\a_{1}+\a_{2}+\a_{3},\ck \a_{2}}\leq 0$ (from $(\a_{1}+\a_{2}+
   \a_{3})-\a_{2}\not\in \Phi\cup\set{0}$). Together, these four formulae give
   $\mpair{\a_{2},\ck\a_{1}}\leq -2$ and  $\mpair{\a_{1},\ck\a_{2}}\leq -2$. In turn, these equations  imply $\mpair{\alpha_{1}+\alpha_{2}, \alpha_{1}+\alpha_{2}}\leq 0$. Since $\mpair{-,-}$ is positive definite, this forces $\alpha_{1}+\alpha_{2}=0$, which contradicts the assumptions.
   
   To prove (c), assume $n\geq 2$ and  choose $\s$ as  in (a).  We use cycle notation for 
   permutations (two-cycles) of $\set{1,\ldots, n}$ below. Note that (a) also holds with $\s$ 
   replaced by 
   $\s':= \s(1,2)$.    By the assumption in (c), one  must  have 
   $1\in \set{\s(1),\s(2)}$ since $\s(1)+\s(2)\in\Phi$. Replacing  $\s$ by $\s'$ if necessary, 
   assume without loss of generality that $\s(1)=1$.  Now for any $i$ with $2\leq i<n$,  let $\b_{i-1}:=\a_{\s(1)}+\ldots+\a_{\s(i-1)}$. Then $\b_{i-1}+\a_{\s(i)}\in \Phi$,  $\b_{i-1}+\a_{\s(i)}
   +\a_{\s(i+1)}\in \Phi$ and 
   $\a_{\s(i)}+\a_{\s(i+1)}\not\in \Phi$. By (b),  $\b_{i-1}+\a_{\s(i+1)}=\a_{\s(1)}+\ldots+\a_{\s(i-1)}+
   \a_{\s(i+1)}\in \Phi$.  This shows that if $\s$ satisfies (a) and  $\s(1)=1$,  then, for each $i$ with $2\leq i<n$, (a) also holds 
   with 
   $\s$ replaced by $\s(i,i+1)$.  Since the adjacent  two-cycles $(i,i+1)$ with $2\leq i< n$ generate the symmetric group on $\set{2,\ldots, n}$,  (a) holds for all 
   permutations $\s$ of  $\set{1,\ldots, n}$ which fix $1$. Now given $I$ as in (c), there is 
   some such  permutation $\s$  and some $j$ with $1\leq j\leq  n$  such that 
   $I=\set{\s(1),\ldots, \s(j)}$, and (c) follows.
      \end{proof} 
      \subsection{}\label{rootperm}  For $J\seq S$, let $\leq_{J}$ 
      denote  the partial order on $V$ such that   $\a\leq_{J} \b$ if $\b-\a\in \Int_{\geq 0}\Pi_{J}$. 
          \begin{lemma}\label{lem:rootperm}
      \begin{num}
      \item If  $\a,\b\in \Phi$ with $\b\in \a+ \real \Pi_{J}$, then there exists $w\in W_{J}$ such that $w(\b)\geq_{J }w(\a)$.  
      \item Let $\a,\b\in \Phi\sm\Phi_{J}$ with $\a\leq_{J} \b$.
      Then one may write $\b-\a=\sum_{i=2}^{n}\alpha_{i}$ with  $\alpha_{i}\in \Phi_{J,+}$  for $i=2,\ldots, n$ and $n\geq 1$ minimal.  Set $\a_{1}:=\a$. Then $\a_{I}\in \Phi$ for all $I\seq L:=\set{1,\ldots, n}$ with $1\in I$.
       \end{num}
      \end{lemma}
      \begin{proof}
      For (a), choose $w\in W_{J}$ so $w(\b-\a)\in \CC_{W_{J}}$.
      Hence $w(\b-\a)\in \CC_{W_{J}}\cap \real \Pi_{J}$. By \cite[Ch V, \S 3,  Lemma 6(i)]{Bour}, one has  $\CC_{W_{J}}\cap \real \Pi_{J}\seq \real_{\geq 0}\Pi_{J}$. Therefore, 
      $w(\b-\a)\in \real_{\geq 0}\Pi_{J}\cap  \Int \Pi=\Int_{\geq 0}\Pi_{J}$  as required. 
      
   For (b), note first that  there exist  expressions  $\b-\a=\sum_{i=2}^{n}\alpha_{i}$ with each $\alpha_{i}\in \Phi_{J,+}$ (since,  using $\b-\a\in \Int_{\geq 0}\Pi_{J}$, they exist even with each $\a_{i}\in \Pi_{J}$). Fix such an expression with  $\alpha_{i}\in \Phi_{J,+}$ and $n$ minimal.   It will suffice to verify  that  $\a_{1},\ldots \a_{n}$ satisfy the hypotheses of Lemma  \ref{lem:rootstring}(c). One has $\a_{L}=\a_{1}+(\a_{2}+\ldots+\a_{n})=\a+(\b-\a)=\b\in \Phi$. Let $\eset \neq I\seq L$.  If $1\in I$, then  $\a_{I}\in \a_{1}+\real \Pi_{J}=\a+\real\Pi_{J}$ 
        and thus $\a_{I}\neq 0$, since  $\a \in\Phi\sm \Phi_{J}$ implies $\a\not\in \real\Pi_{J}$.
        If $1\not\in I$, then $\a_{I}\neq 0$ by minimality of $n$ (for instance). Finally, if $2\leq i<j\leq n$, then $\a_{i}+\a_{j}\not \in \Phi_{J}$ by minimality of $n$, and hence $\a_{i}+\a_{j}\not\in \Phi$
        since $\a_{i}+\a_{j}\in \real\Pi_{J}$ and $\real\Pi_{J}\cap \Phi=\Phi_{J}$.
       \end{proof}
       \subsection{Proof of \ref{prop:osh}(a) for  Weyl groups}\label{oshWeyl} 
 By rescaling if necessary, we may assume $\Phi$ is crystallographic. We may also assume that $W$ is irreducible. Then by
        \cite[Ch VI, \S 1, Proposition 11]
{Bour}, $\Phi$ has at most two root lengths. Rescaling again to replace $\Phi$ by $\ck \Phi$ if necessary, we may assume that $\a$ is a long root.
To simplify the argument, we assume using Lemma \ref{lem:oshdihed} that $\Phi$ is not of type $G_{2}$ and  multiply the inner product on $V$ by a positive 
scalar (or do a further rescaling) if necessary   so that the maximal squared length of a root in $\Phi$ is $\mpair{\a,\a}=4$. Roots of squared length $4$ are 
here called long roots;
other roots (if there are any) are called short roots, and ( cf. \cite{Bour}) they have squared length $2$.  The roots of any fixed length form a single $W$-orbit.

Let $\a\in \Phi\sm \Phi_{J}$ and $\b\in W\a\cap (\a+\real\Pi_{J})$ i.e. 
$\b\in \Phi\cap (\a+\real\Pi_{J})$ with $\mpair{\b,\b}=4$. We have to show
$\b\in W_{J}\a$. Let $w\in W_{J}$ be as in Lemma \ref{lem:rootperm}(a). It will suffice to show that $w(\b)\in W_{J}w(\a)$. Replacing $\a$ by $w(\a)$ and $\b$ by $w(\b)$, we assume without loss of generality that $\a\leq_{J}\b$.
Let $\a_{1},\ldots,\a_{n}\in \Phi_{J,+}$ be as in Lemma \ref{lem:rootperm}(b),
so $\a_{I}\in \Phi\cap (\a+\real\Pi_{J})$  for all $I\seq L$ with $1\in I$.
If $n\leq 2$, then $\b\in W_{J}\a$ by Lemma \ref{lem:oshdihed}. We assume that $n\geq 3$ and   prove that $\b\in W_{J}\a$ by induction on $n$. 
If  $\a_{I}\in \Phi$ is a long root for some $I$ with $1\in I\sneq L$, then by induction,  $\a$ is $W_{J}$-conjugate to $\a_{I}$ and $\a_{I}$ is $W_{J}$-conjugate to $\b$, so $\b$ is $W_{J}$-conjugate to $\a$.
Hence we may assume that  all $\a_{I}$ with $1\in I\sneq L$ are short roots. 
(In particular, this completes the  argument  if $\Phi$ is simply laced, since there are no short roots in that case).  

Note that, writing $I':=I\sm\set{1}$ for $I\seq L$ with $1\in I$, one has 
\begin{equation}\label{eq:lengths}
\mpair{\a_{I},\a_{I}}=\mpair{\a_{1},\a_{1}}+\sum_{j\in I'}\mpair{\a_{j},\a_{j}}+2\sum_{j\in I'}\mpair{\a_{1},\a_{j}}+2\sum_{i,j\in I':i<j}\mpair{\a_{i},\a_{j}}
\end{equation}
First, take $I=\set{1,i}$  in this where $2\leq i\leq n$. It gives $2=4+\mpair{\a_{i},\a_{i}}+2\mpair{\a_{1},\a_{i}}$.  If $\mpair{\a_{i},\a_{i}}=4$,  then $\mpair{\a_{1},\ck\a_{i}}\not\in \Int$, a contradiction. Hence $\mpair{\a_{i},\a_{i}}=2$ and $\mpair{\a_{1},\a_{i}}=-2$.  Hence each $\a_{i}$ with $2\leq i\leq n$ is a short root, 
and \eqref{eq:lengths} simplifies (for general $I\ni 1$) to 
\begin{equation}\label{eq:lengths2}
\mpair{\a_{I},\a_{I}}=4-2\vert I'\vert +2\sum_{i,j\in I':i<j}\mpair{\a_{i},\a_{j}}.
\end{equation}
Now take $I=\set{1,2,3}$ in this. One gets $\mpair{\a_{I},\a_{I}}=2\mpair{\a_{2},\a_{3}}$. Suppose first that $\mpair{\a_{I},\a_{I}}=4$ 
i.e. $\a_{I}$ is long. This forces $\a_{I}=\b$, so $I=\set{1,2,\ldots, n}$ (i.e $n=3$). It also gives   $\mpair{\a_{2},\a_{3}}=2$  and therefore $\a_{2}=\a_{3}$ since $\a_{2}$ and $\a_{3}$ are short.
In this case, we get \begin{equation*}
\b=\a_{I}=\a_{1}+\a_{2}+\a_{3}=\a_{1}+2\a_{2}=s_{\a_{2}}(\a_{1})\in W_{J}(\a)
\end{equation*} as required. The other possibility is that $\mpair{\a_{I},\a_{I}}=2$  i.e. $\a_{I}$ is short.  Then $\b\neq \a_{I}$ so $n\geq 4$, and 
 $\mpair{\a_{2},\a_{3}}=1$. One has $s_{\a_{2}}(\a_{3})=\a_{3}-\a_{2}\in \Phi_{J}$, so one of $\pm s_{\a_{2}}(\a_{3})$ is in $\Phi_{J,+}$. Interchanging 
 $\a_{2}$ and $\a_{3}$ if necessary, suppose without loss of generality that $\a_{3}':=\a_{3}-\a_{2}\in \Phi_{J,+}$.
Then  $\a_{2}':=s_{\a_{2}}(\a_{1})=\a_{1}+2\a_{2}\in W_{J}\a$.
Also, $\b=\a_{1}+\a_{2}+\a_{3}+\ldots +\a_{n}=\a_{2}'+\a_{3}'+\a_{4}+\ldots +\a_{n}$ where $\a_{3}',\a_{4},\ldots,\a_{n}\in \Phi_{J,+}$. Note  that $\a_{2}'\in \Phi\sm \Phi_{J}$ has
 square length  $\mpair{\a_{1},\a_{1}}=\mpair{\b,\b}$. By induction, 
 $\b\in W_{J}\a_{2}'=W_{J}\a$ as required. This completes the proof of Proposition \ref{prop:osh}(a)
 for crystallographic $W$.

\section{Proof of \ref{prop:osh} in the remaining cases (types $H_{3}$ and $H_{4}$)}
\label{s5}
The arguments  above suffice to establish 
Proposition \ref{prop:osh} provided that $(W,S)$ has no irreducible 
components of type $H_{3}$ or $H_{4}$.
To complete the proof in general, it is enough to establish  that it holds if $(W,S)$ is of type $H_{4}$ or $H_{3}$
 (in fact, by Lemma \ref{lem:oshequiv}(e), the case $H_{4}$ suffices). 
We do this using a suitable  embedding of $W$ as a subgroup of a Weyl group of   simply laced ($A$, $D$, $E$) type, as discussed in detail in \cite[\S1]
{DyEmb1} for  the case of $H_{4}$ in $E_{8}$.

\subsection{The non-crystallographic case-preparation} \label{noncrystsetup}  
  Define the number ring 
 $R:=\Int[\tau]\seq \rat(\sqrt 5)\seq \real$
  where $\tau:=\frac{1+\sqrt{5}}{2}$ denotes the golden ratio, so $\tau^{2}=\tau+1$.
Let $\Phi$   be a finite root system for a Coxeter group $W$  such that 
$\mpair{\a,\a}=2$ for all $\a\in \Pi$ and $\mpair{\a,\b}\in \set{0,-1,-\tau}$ for all distinct $\a,\b\in \Pi$. Any finite Coxeter  group $W$ such that the off diagonal entries of its Coxeter matrix are all $2$, $3$ or $5$ has such a root system, in which, for distinct $\a,\b\in \Pi$, $s_{\a}s_{\b}$ has order $2$, $3$ or $5$ according as whether $\mpair{\a,\b}$ is equal to $0$, $-1$ or $-\tau$.
Any root  system  of type $H_{4}$ or $H_{3}$ can be rescaled if necessary to be this form. 

  Let   $U:=R\Pi$ be the free $R$-submodule  of $V$ 
with $R$-basis $\Pi$. One has $\Phi\seq U$. Regard $W$ as 
a  group acting faithfully on $U$, generated by $R$-linear  reflections  $s_{\a}\colon U\to U$ for $\a\in \Phi$ (or just in $\Pi$) given by
 $v\mapsto v-\mpair{v,\a}\a$ for $v\in U$, where $\mpair{-,-}$ is now 
 regarded as a symmetric $R$-bilinear map $U\times U\to R$.     

Define a $\Int$-linear map $\th\colon R\to \Int$ by $\th(a+b\tau)=a$ for 
$a,b\in \Int$ and a symmetric  $\Int$-bilinear form 
$(-,-)\colon U\times U\to \Int$ by $(\a,\b)=\th(\mpair{\a,\b})$. Then as in \cite[\S1]{DyEmb1}, 
$\D:=\Pi\dotcup \tau\Pi$ is a simple system for a simply laced  finite crystallographic root system 
$\Psi:=\Phi\dotcup \tau\Phi$, with $(\a,\a)=2$ for all $\a\in \Phi$ and with $\Int$-linear  reflections 
$r_{\a}\colon U\to U$ given by $v\mapsto v- (v,\a)\a$ for $\a\in \Psi$ (note $U$ is the root lattice of $\Psi$). 
Denote the corresponding finite Coxeter system as $(W',S')$.
For $\a\in \Phi$, one has $s_{\a}=r_{\a}r_{\t\a}=r_{\t\a}r_{\a}$.  In particular, $W\seq W'$. The inclusion map $W\to W'$ is length doubling:
for $w\in W$, $l'(w)=2l(w)$ where $l'$ (resp., $l$) is the length function of $(W',S')$ (resp., $(W,S)$).  
\subsection{}\label{noncryst} A set of roots of $\Psi$ of the form $\set{\a,\t\a}$ where $\a\in \Phi$ will be called  a \emph{bundle}. For the proof of Proposition \ref{prop:osh} in the remaining cases, we need only (b) from the following, but the other parts are also  of  interest.
\begin{prop}\label{prop:noncryst}
\begin{num}
\item   The map 
$\wh \varphi\colon \wp(\Phi)\to \wp(\Psi)$ of power sets given by 
$\G\mapsto \G\dotcup \tau\G$ restricts to  a $W$-equivariant bijection
$\varphi$  from the set of simple subsystems of $\Phi$ to the set  of 
simple  subsystems    of $\Psi$ which are unions of bundles. 
\item  If $\G$ is a simple subsystem of $\Phi$, $\G':=\varphi(\G)$ and
$\a\in \Phi$, then $\a\in  \CC_{W_{\G}}$ if and only if $\t\a\in
 \CC_{W'_{\G'}}$.
\item If $\G$ is a simple subsystem of $\Phi$ and $\G':=\varphi(\G)$,  the Coxeter graph of $\G'$ is obtained by replacing each component of the Coxeter graph of $\G$ of type $H_{4}$ (resp., $H_{3}$, $I_{2}(5)$) by one of type $E_{8}$, $D_{6}$, $A_{4}$) and replacing each simply laced component of $(W,S)$ by two components of the same type.
\item The map
$\varphi$ induces a $W$-equivariant  bijection $\varphi'$ from the set of root subsystems of $\Phi$ to the set of root subsystems of $\Psi$ which are unions of bundles, with $\varphi'(\L)=\L\dotcup \tau\L$. 
\end{num}
\end{prop} 
\begin{proof}
Given $\a,\b\in \Phi$, there exists $w\in W$ such that $w\a,w\b$ 
are in a  standard parabolic subgroup of $W$ of rank at most two. By an 
easy case by case check of the possibilities (which are  types $A_{1}$, $A_{1}\times A_{1}$, $A_{2}$ or $I_{2}(5)$), one sees that $\mpair{\a,\b}\in \set{0,\pm 1, \pm (\tau-1), \pm \tau,\pm 2}$.  One has the following possibilities for inner products:

\begin{center}
\begin{tabular}{|c| |c|c|c|c|}\hline
$\mpair{\a,\b}$&$(\a,\b)$&$(\a,\t\b)$&$(\t\a,\b)$&$(\t\a,\t\b)$\\ \hline
$0$ & $0$&$0$&$0$&$0$\\ \hline
$\pm 1$ & $\pm 1$&$0$&$0$&$\pm1$\\ \hline
$\pm (\t-1)$ & $\mp 1$&$\pm 1$& $\pm 1$& $0$\\ \hline
$\pm \tau$ &$0$&$\pm 1$&$\pm 1$&$\pm 1$\\ \hline
$\pm 2$ & $\pm 2$&$0$&$0$&$\pm2$\\ \hline
\end{tabular}\end{center}

Note that (using \cite{DyRef} for instance) a subset $\G$ of $\Phi$ is  a simple subsystem if  and only if it 
is  $R$-linearly 
independent and for any distinct $\a,\b\in \G$, one has 
$\mpair{\a,\b}\in \set{0,-1,-\t}$. Similarly, a subset $\G'$ of $\Psi$ is a simple subsystem if  and only if it 
is $\Int$-linearly 
independent and for any distinct $\a,\b\in \G'$, one has 
$\mpair{\a,\b}\in \set{0,-1}$. Hence (a) follows by inspection of the above table.  Note next that the set of  non-negative  inner products
 $\mpair{\a,\b}$ with $\a,\b\in \Phi$ is contained in  $\set{0,1,\tau-1,\tau,2}$. Part  (b) follows immediately from the table and the definition of  
 fundamental chamber. 
 
   It is enough to prove (c) for irreducible $\Phi$. 
 For $\Phi$ of type $H_{4}$, the result is clear from  \cite[\S1]
{DyEmb1}.  The cases of  $H_{3}$, $I_{2}(5)$ follow,  since they are  induced embeddings of  standard parabolic subgroups of $H_{4}$ in standard parabolic subgroups of $E_{8}$. The simply laced cases hold by inspection of the  table. 

It remains to prove (d). Let $\L$ be a root subsystem of $\Phi$ and let $\G$ be a simple subsystem of $\L$. 
The discussion of \ref{noncrystsetup} applied to $(\L,\G)$ instead of 
$(\Phi,\Pi)$  implies  that  $\varphi'(\L):=\L\cup\tau\L$ 
(which is obviously  a union of bundles) is a root subsystem of $\Psi$   with a simple system $\varphi(\G)=\G\cup \tau(\G)$. Obviously the map $\varphi'$ so defined is $W$-equivariant, and it is injective since
$\Phi\cap \varphi'(\L)=\L$.  To complete the proof of (d), take any root subsystem $\L'$ of $\Psi$ which is a union of bundles.
Let $\L:=\L'\cap \Phi$. For $\a,\b\in \L$, one has $s_{\a}(\b)=r_{\a}r_{\tau \a}(\b)\in \Phi\cap \L'=\L$ since $\a,\tau\a\in \L'$. Hence $\L$ is a root subsystem of $\Phi$, and clearly $\varphi'(\L)=\L'$.  \end{proof}
 
\subsection{Proof of \ref{prop:osh}(a) for $H_{3}$ and $H_{4}$}\label{oshend}
It is now easy to deduce the validity of  Proposition \ref{prop:osh}(b)
(which is equivalent to \ref{prop:osh}(a))
for $(W,S)$ of type $H_{4}$ or $H_{3}$ from its validity for simply laced root systems  as follows. Without loss of 
generality, take  $(W,S,\Phi, \Pi)$ as in  \ref{noncrystsetup} and define $(W', S',\Psi,\D)$ as there. Fix $J\seq S$ and let $J'\seq S'$ with 
$\D_{J'}=\Pi_{J}\dotcup \t\Pi_{J}$.
Let $\a,\b\in (\Phi\sm \Phi_{J})\cap \CC_{W_{J}}$ with
$\b\in \a+\real\Pi_{J}$.  Since $\Phi\seq R\Pi$, this gives 
$\b\in \a+R \Pi_{J}=\a+\Int\D_{J'}$. 
Note also that $\Phi_{J}=\Phi\cap R\Pi_{J}$. Since  $\a,\b\in \Phi\sm \Phi_{J}$ and $\tau$ is a unit in $R$, one has
 $\t\a,\t\b\not\in R\Pi_{J}=\Int\D_{J'}$ and so 
 $\t\a,\t\b\in \Psi\sm \Psi_{J'}$. By Proposition   \ref{prop:noncryst}(b),  one has $\t\a,\t\b\in \CC_{W'_{J'}}$. By Lemma \ref{prop:osh}(b) for $(W',S',\Phi',\D')$, which is already known from Section \ref{s4} since  $\Psi$ is crystallographic, it follows that $\t\a=\t\b$ and thus $\a=\b$. This proves \ref{prop:osh}(b) for  $(W,S,\Phi,\Pi)$ and completes the proof of  Proposition \ref{prop:osh}(a), and hence all results of this paper, in all cases.

\end{document}